\documentclass[11pt]{article}
\RequirePackage{fix-cm}
\usepackage{graphicx}
\usepackage{amsmath}
\usepackage{amsthm}
\usepackage{amsfonts}
\usepackage{amssymb}
\usepackage{graphicx}
\usepackage{color}
\usepackage{rotating}
\usepackage{authblk}
\usepackage{url}
\usepackage{tcolorbox}
\usepackage{tikz,pgfplots} 
\usetikzlibrary{arrows,shadings,shadows,automata,shapes,calc,positioning,patterns,decorations.markings,decorations.pathmorphing,snakes,plotmarks}
\usepackage{soul}
\usepackage{footnote}
\usepackage{geometry}
\usepackage{hyperref}
\usepackage{authblk}
\usepackage{lipsum}

\newcommand{\Z}{\mathbb{Z}}
\newcommand{\R}{\mathbb{R}}
\newcommand{\Q}{\mathbb{Q}}

\newcommand{\seq}{\subseteq}

\DeclareMathOperator{\conv}{conv}
\DeclareMathOperator{\cone}{cone}

\DeclareMathOperator{\rank}{rank}

\renewcommand{\S}{\mathcal{S}}
\newcommand{\Sk}{\S^{k}}
\renewcommand{\L}{\mathcal{L}}
\newcommand{\Mset}{M}
\newcommand{\M}{\mathcal{\Mset}}
\newcommand{\Tset}{T}

\newcommand{\ones}{\boldsymbol{1}}
\newcommand{\zeros}{\boldsymbol{0}}

\newcommand{\floor}[1]{\lfloor#1\rfloor}
\newcommand{\ceil}[1]{\lceil#1\rceil}
\newcommand*{\medcap}{\mathbin{\scalebox{1.5}{\ensuremath{\cap}}}}%

\newcommand{\bpar}{\beta}

\newtheorem{theorem}{Theorem}
\newtheorem{proposition}{Proposition}
\newtheorem{lemma}{Lemma}

\newtheorem{corollary}{Corollary}
\newtheorem{observation}{Observation}

\newcommand{\cred}{\color{black}}

\newcommand{\zLPc}{z^{LP}(c)}
\newcommand{\zIc}{z^{I}(c)}

\makeatletter
\renewcommand\footnoterule{%
  \kern-3\p@
  \hrule\@width \textwidth
  \kern2.6\p@}
\makeatother

\makeatletter
\renewcommand*{\@fnsymbol}[1]{\ensuremath{\ifcase#1\or *\or \dagger\or \ddagger\or **\or \mathsection\or \mathparagraph\or \|\or  \dagger\dagger
   \or \ddagger\ddagger \else\@ctrerr\fi}}
\makeatother

\newcounter{mynotes}
\def\conv{\mathop{\rm conv}}
\def\cone{\mathop{\rm cone}}

\def\int{\mathop{\rm int}}

\def\spann{\mathop{\rm span}}

\begin{document}

\title{Lower bounds on the lattice-free rank for packing and covering integer programs}

\author[1]{Merve Bodur\thanks{bodur@mie.utoronto.ca}}
\author[2]{Alberto Del Pia\thanks{delpia@wisc.edu}}
\author[3]{Santanu S. Dey\thanks{santanu.dey@isye.gatech.edu}}
\author[4]{Marco Molinaro\thanks{mmolinaro@inf.puc-rio.br}} 
\affil[1]{\small Department of Mechanical and Industrial Engineering, University of Toronto}
\affil[2]{\small Department of Industrial and Systems Engineering \& Wisconsin Institute for Discovery, University of Wisconsin-Madison}
\affil[3]{\small School of Industrial and Systems Engineering, Georgia Institute of Technology} 
\affil[4]{\small Computer Science Department, Pontifical Catholic University of Rio de Janeiro}

\maketitle

\begin{abstract}
In this paper, we present lower bounds on the rank of the split closure, the multi-branch closure and the lattice-free closure for packing sets as a function of the integrality gap. We also provide a similar lower bound on the split rank of covering polyhedra. These results indicate that whenever the integrality gap is high, these classes of cutting planes must necessarily be applied for many rounds in order to obtain the integer hull.
\\ \\
\smallskip
\noindent \textbf{Keywords.} Integer programming, packing, covering, split rank, multi-branch split rank, lattice-free rank

\end{abstract}
\section{Introduction}
\label{sec:intro}

\emph{Split cuts} are a very important class of cutting planes in integer programming both from a theoretical and computational perspective (see for example \cite{balas:1979, BalasS08, cook:ka:sc:1990}). Recently, many generalizations of split cuts have been studied, such as the \emph{multi-branch split cuts} \cite{dash2014lattice, dash2013t, li:ri:2008} and the \emph{lattice-free cuts} \cite{andersen2007inequalities,basu2015geometric,borozan:2007,RichardDey}. In order to study the strength of the cutting plane procedures, a very useful concept is the notion of \emph{rank} which represents the minimum rounds of cuts needed to obtain the integer hull. The notion of rank was first studied in the context of Chv\'atal-Gomory (CG) cuts \cite{Schrijver80}. Many lower bounds on the rank of the above mentioned closures have been proven; see \cite{BasuCM12,bodur2017cutting,cook:ka:sc:1990,dey:lowerbnd:2009,DeyL11,li:ri:2008} for the split rank, see \cite{dash2013t} for the multi-branch rank, and see \cite{averkov2017approximation} for the lattice-free rank. 

A standard notion describing the {\cred difficulty} of an integer program is the \emph{integrality gap} which in this paper refers to the ratio between the optimal objective function values of the integer program and its linear programming relaxation. While it is natural to expect that the rank of a cutting plane procedure should increase with the increase in the integrality gap, only a few results of this nature exist in the literature \cite{bodur2016aggregation,PokuttaS11}. 

In this paper, we present lower bounds on the rank of the split closure, the multi-branch closure and the lattice-free closure for \emph{packing sets} as a function of the integrality gap. We also provide a similar lower bound on the split rank of \emph{covering polyhedra}. These results indicate that whenever the integrality gap is high, these classes of cutting planes must necessarily be applied for many rounds in order to obtain the integer hull.

The rest of the paper is organized as follows. We provide all necessary definitions in Section \ref{sec:Prelim}. We state all our main results in Section \ref{sec:Statements}. 
Finally, in Section \ref{sec:Packing} and Section \ref{sec:Covering} we present the proofs for results concerning the packing and covering cases, respectively.
\section{Preliminaries}
\label{sec:Prelim}
For an integer $t \geq 1$, we use $[t]$ to describe the set $\{1, \dots, t\}$. Also, we represent the $j^{\text{th}}$ unit vector, the vector of ones and the vector of zeros in appropriate dimension by $e_j$, $\ones$ and $\zeros$, respectively. Given a set of vectors $v^1, \hdots, v^t$, we denote the linear subspace spanned by these given vectors as $\spann ( \{ v^j \}_{j \in [t]} )$.

\paragraph{Sets.} In this paper, we work with \emph{covering {\cred polyhedra}} and \emph{packing {\cred polyhedra}}, which are of the form 
$$P_C = \{x \in \R^n_+ \mid Ax \ge b \} \quad \text{and} \quad P_P = \{x \in \R^n_+ \mid Ax \le b \},$$
respectively, where all the data $(A,b) \in \Q_+^{m \times n} \times \Q_+^m$. {\cred Thus, an inequality of packing type (respectively covering type) is one of the form $a^\top x \le b$ (respectively $a^\top \ge b$) for non-negative $(a,b) \in \Q_+^n \times \Q_+$.} If it is obvious from the context that the polyhedron is of covering (resp. packing) type, we may drop the subscript $C$ (resp. $P$). For the packing case, we also work with more general sets. We call $Q \ {\cred \subseteq} \ \R_+^n$ a \emph{packing set} if $x \in Q$ and $\zeros \leq y \leq x$ imply that $y \in Q$. 

Throughout this paper, we make a technical assumption regarding the sets under consideration that we call as \emph{well-behavedness}. The set $P_C$ is \emph{well-behaved} if $A_{ij} \leq b_i$ for all $i \in [m], j \in [n]$. Notice that this is a natural assumption since if $A_{ij} > b_i$ for some $i \in [m], j \in [n]$, then we can replace the coefficient $A_{ij}$ by $b_i$ to obtain a tighter linear programming relaxation with the same set of feasible integer points. A packing set $Q$ is \emph{well-behaved} if $e_j \in Q$ for all $j \in [n]$. This is not a restrictive assumption since if $e_j \notin Q$ for some $j \in [n]$, then we replace $Q$ with the packing set $\{ x \in Q \mid x_j = 0\}$, which provides a tighter linear relaxation with the same set of feasible integer points. Note that if $Q$ is the polytope $P_P$, then {\cred the} well-behavedness definition is equivalent to $A_{ij} \leq b_i$ for all $i \in [m], j \in [n]$.

{\cred  A \emph{relaxation} of a set $P$ is any superset $\tilde{P} \supseteq P$.} Let $\alpha > 0$ be a scalar. If a given covering polyhedron $\tilde{P}_C$ is a relaxation of $P_C$ and satisfies
$$\min\{c^\top x \mid x \in \tilde{P}_C\} \ge \frac{1}{\alpha} \cdot \min\{c^\top x \mid x \in P_C\}, \ \ \forall c \in \R_+^n,$$
then $\tilde{P}_C$ is an \emph{$\alpha$-approximation} of $P_C$. Similarly, given a packing set $P_P$ and one of its relaxation $\tilde{P}_P$ of packing type, $\tilde{P}_P$ is an \emph{$\alpha$-approximation} of $P_P$ if 
$$\max\{c^\top x \mid x \in \tilde{P}_P\} \le \alpha \cdot \max\{c^\top x \mid x \in P_P \}, \ \ \forall c \in \R_+^n.$$
For a set $P \subseteq \R^n$, we define $\alpha P:= \{\alpha x \mid  x \in P\}$. 
The equivalent definitions of $\alpha$-approximation for covering and packing cases are provided in \cite{bodur2016aggregation} as
$$\frac{1}{\alpha} \tilde{P}_C  \seq P_C \quad \text{and} \quad \tilde{P}_P \seq \alpha P_P,$$
respectively.

Given a polyhedron $P \seq \R^n$, we denote its integer hull by $P^I := \conv( \{x \mid x \in P \cap \Z^n\})$ where $\conv(\cdot)$ is the convex hull operator. We let $\zLPc$ and $\zIc$ to denote the optimal value of a given objective function $c^\top x$ over $P$ and $P^I$, respectively. For convenience, we will sometimes refer to $\zLPc$ and $\zIc$ as $z^{LP}$ and $z^{I}$, respectively.

\paragraph{Closures.}
We call a set $\Mset \in \R^n$ a \emph{strict lattice-free set} if $M \cap \Z^n = \emptyset$. Note that the set $\Mset$ need not to be convex. Given a set $Q$, one can obtain a relaxation of $Q^I$ as 
$$Q^\Mset := \conv(Q \setminus \Mset).$$  
Given a collection of strict lattice-free sets $\M$, we define the corresponding closure as
$$\M(Q) = \bigcap_{\Mset \in \M} Q^\Mset.$$
For convenience, we sometimes refer to $\M$ as the closure operator or just as closure.

Next, we define three special cases of the strict lattice-free closures, namely the split closure, the multi-branch closure and the lattice-free closure.

We denote the \emph{split set} associated with $(\pi,\pi_0) \in \Z^n \times \Z$ by
$$S(\pi,\pi_0) := \{ x \in \R^n \mid \pi_0 < \pi^\top x < \pi_0+1\}.$$
{\cred Denoting} the collection of all split sets by 
$$
{\cred \S = \{S(\pi,\pi_0) \mid (\pi,\pi_0) \in \Z^n \times \Z\},}
$$
the split closure of $Q$, denoted as $\S(Q)$, is defined to be 
$$\S(Q) = \bigcap_{S \in \S} Q^S.$$
For convenience, we denote $Q^{S(\pi,\pi_0)}$ by $Q^{\pi,\pi_0} $ which is explicitly defined as 
$$Q^{\pi,\pi_0} = \conv(Q \setminus S(\pi,\pi_0)) = \conv \big( (Q \cap \{ \pi^\top x \leq \pi_0 \}) \cup (Q \cap \{ \pi^\top x \geq \pi_0+1 \}) \big).$$


A generalization of split closure, called the \emph{$k$-branch split closure}, which is defined by \cite{li:ri:2008}, is obtained by removing the union of \emph{at most} $k$ split sets simultaneously. Letting
\begin{align*}
Q^{\pi^1,\dots, \pi^k;\pi^1_0, \dots, \pi^k_0} :&= \conv \Big( Q \setminus \bigcup_{i \in [k]} S(\pi^i,\pi^i_0) \Big) \\
&= \conv \left( \bigcap_i (Q \cap \{ (\pi^i)^\top x \leq \pi^i_0 \}) \cup (Q \cap \{ (\pi^i)^\top x \geq \pi^i_0+1 \}) \right),
\end{align*}
the $k$-branch split closure of $Q$, denoted by $\Sk(Q)$, can be written as
$$\Sk(Q) = \bigcap_{(\pi^i,\pi^i_0) \in \Z^n \times \Z,~i \in [k]} Q^{\pi^1,\dots, \pi^k;\pi^1_0, \dots, \pi^k_0}.$$
Note that the 1-branch split closure is equivalent to the split closure, i.e, $\S^{1}(Q)=\S(Q)$.
 
A further generalization of the split closure is the so-called \emph{lattice-free closure}, which is obtained by considering convex sets having no integer point in their interior; see \cite{dash2012two,dash2014lattice} for relations to the $k$-branch split closure. A set $L \seq \R^n$ is called a \emph{lattice-free set} if $\int(L) \cap \Z^n = \emptyset$ where $\int(\cdot)$ is the interior operator. For each  integer $k \geq 2$, we define $\L^k$ as the family of full-dimensional lattice-free polyhedra $L \subset \R^n$ defined by \emph{at most} $k$ inequalities. (Note that it is not possible to have lattice-free sets defined by only one inequality.) We denote the \emph{$k$-lattice-free closure} of $P$ by $\L^k(Q)$, i.e.,
$$\L^k(Q) = \bigcap_{L \in {\cred \L^k}} Q^L,$$
where 
$$Q^L := \conv (Q \setminus \int(L)).$$



Given a closure operator $\M$ and a nonnegative objective function $c \in \R^n_+$, we use $z^{\M}$ to denote the optimal value of the minimization (or  maximization) of $c^\top x$ over the closure $\M(Q)$. Lastly, we define the \emph{rank} of the closure $\M$, denoted by $\rank_{\M}(Q)$, as the minimum number of iterative applications of $\M$ to obtain the integer hull of $Q$. We note that the split rank, thus the multi-branch rank and the lattice-free rank, are finite whenever $Q$ is a rational polyhedron or is a bounded set \cite{Schrijver80}. 
\section{Main results}
\label{sec:Statements}
\subsection{Packing} 
The main proof strategy to prove lower bounds on ranks of various cutting plane closures is presented in the proposition below. 

\begin{proposition}
\label{thm:thm1}
Let $\M$ be a collection of strict lattice-free sets which satisfies the following two conditions:
\begin{enumerate}
\item Packing invariance: For any packing set $Q$, $\M(Q)$ is a packing set.
\item Constant approximation: There exists $\alpha_\M \geq 1$ such that $Q \seq \alpha_\M \M(Q)$ for every well-behaved packing set $Q$.
\end{enumerate}

Then, for any well-behaved packing set $Q$, 
$$\rank_{\M}(Q) \geq {\cred \sup_{c \in \R_+^n}} \left\lceil \frac {\log_2 \left( \frac{\zLPc}{\zIc}\right)}{\log_2 \alpha_\M}\right\rceil.$$
\end{proposition}

The proof of Proposition \ref{thm:thm1} is based on a simple iterative argument, which is provided in Section~\ref{subsec:4.1}. 

\subsubsection{Tools to prove the assumptions of Proposition \ref{thm:thm1}}
In order to use Proposition \ref{thm:thm1}, we need to verify the packing invariance and constant approximation properties. The next tool is very helpful in proving packing invariance.

\begin{theorem}
\label{thm:thm2}
Let $\M$ be a collection of strict lattice-free sets. For $\Tset \seq [n]$, define $H[\Tset] := \{ x \in \R^n \mid x_j = 0, \ \forall j \in T\}$. Given $\Mset \in \M$, let 
$$M[\Tset] := ( M \cap H[\Tset] ) + \spann ( \{ e_{j} \}_{j \in T} ). $$ 
Suppose that $\M$ satisfies the following property: For any $\Mset \in \M$ and $T \seq [n]$, $M[\Tset] \neq \emptyset$ implies that $M[\Tset] \in \M$. Then $\M$ is packing invariant. 
\end{theorem}

Note that it is straightforward to see that the set $\Mset[T]$ in Theorem \ref{thm:thm2} is guaranteed to be a strict lattice-free set by construction. The proof of Theorem \ref{thm:thm2} is essentially based on the fact that a cut generated using a strict lattice-free set $M$ is dominated by a packing type inequality that is obtained using the strict lattice-free set $M[T]$ for a specifically chosen set $T$. The details of the proof of Theorem \ref{thm:thm2} are given in Section~\ref{subsec:4.2}.

We observe here that in order to use Proposition \ref{thm:thm1}, we must prove {\cred the} constant approximation property for general well-behaved packing sets, rather than just for polyhedra. The reason is that the closures of some cutting plane families we consider are not known to be polyhedral. In order to prove  {\cred the} constant approximation property for general well-behaved packing sets, we will find it convenient to prove this property first for well-behaved packing polyhedra. It turns out that this is sufficient to prove constant approximation property for any well-behaved packing set as the next theorem states.

\begin{theorem}
\label{thm:thm3}
Let $\M$ be a collection of strict lattice-free sets with the following property: There exist $\alpha_\M \geq 1$ such that $P_P \seq \alpha_\M \M(P_P)$ for every well-behaved packing polyhedron $P_P$. Then, $Q \seq \alpha_\M \M(Q)$ for every well-behaved packing set $Q$.
\end{theorem}

Theorem \ref{thm:thm3} is proven by first constructing a well-behaved packing polyhedron which is an inner approximation of $Q$ and is arbitrarily close to $Q$. We then show how to ``transfer" the $\alpha_\M$ factor from this polyhedron to $Q$. The details of the proof of Theorem \ref{thm:thm3} are provided in Section~\ref{subsec:4.3}. 

\subsubsection{Applications of Proposition \ref{thm:thm1} to split, multi-branch split and lattice-free closures}

We use Theorem \ref{thm:thm2} to verify the following result.

\begin{theorem}[Packing invariance]
\label{thm:thm4}
$\M$ is packing invariant for $\M \in \{ \S, \S^k, \L^k \}$.
\end{theorem}

Theorem \ref{thm:thm4} is proven in Section~\ref{subsec:4.4}.

\begin{theorem}[Constant approximation]
\label{thm:thm5}
For $\M \in \{ \S, \S^k, \L^k \}$, $\M$ satisfies the constant approximation property, where {\cred we can choose} $\alpha_{\S} = 2$, {\cred $\alpha_{\S^k} = \min \{ 2^k,n \}+1$, and $\alpha_{\L^k} = \min \{ k,n \}+1$}. 

Moreover, the factor $\alpha_{\S}$ is tight, i.e., for every $\epsilon >0$, there exists a well-behaved packing polyhedron $\tilde{P}_P$ such that $\tilde{P}_P \not\subseteq (2-\epsilon) \S(\tilde{P}_P)$. 
\end{theorem}

Observe that the split cuts are a special case of multi-branch split cuts. However, we have stated their constant approximation result separately since the general factor for multi-branch split closure is not tight for the split closure. Indeed, proving the factor of 2 in the case of split cuts involves more careful analyses. Moreover, this factor of 2 for the split case is tight as stated in the theorem. The proof of Theorem \ref{thm:thm5} for the split, multi-branch split and lattice-free cases are given in Sections \ref{subsubsec:4.5.1}, \ref{subsubsec:4.5.2} and \ref{subsubsec:4.5.3}, respectively.

Note that a factor of $2$ is proven in \cite{bodur2016aggregation} as an approximation factor of the \emph{aggregation closure}, which is very similar to the result {\cred for the} split closure in Theorem \ref{thm:thm5}. However, the split closure result of Theorem \ref{thm:thm5} is not implied by the result of \cite{bodur2016aggregation} since for packing polyhedra, split cuts are not {\cred always} dominated by \emph{aggregation cuts}, see the example given in Observation \ref{obs:PackingSplitAgg} in Appendix \ref{sec:appendix}.

Proposition \ref{thm:thm1}, Theorem \ref{thm:thm4} and Theorem \ref{thm:thm5} lead us to the following lower bounds on the rank of {\cred the} split closure, $k$-branch split closure and $k$-lattice-free closure of packing sets. As Corollary \ref{cor:cor1} is a direct application of Proposition \ref{thm:thm1}, we omit its proof.

\begin{corollary}
\label{cor:cor1}
Let $Q$ be a well-behaved packing set. Then
\begin{enumerate}
\item $\rank_{\S}(Q) \geq {\cred \sup_{c \in \R_+^n}} \left\lceil \log_2 \left( \frac{\zLPc}{\zIc}\right)\right\rceil$.
\item $\rank_{\S^k}(Q) \geq {\cred \sup_{c \in \R_+^n}} \left\lceil\frac{\log_2 \left( \frac{\zLPc}{\zIc}\right)}{\log_2 (\min \{ 2^k,n \}+1)}\right\rceil$ for any $k \in \Z_+, k \geq 1$.
\item $\rank_{\L^k}(Q) \geq {\cred \sup_{c \in \R_+^n}} \left\lceil\frac{\log_2 \left( \frac{\zLPc}{\zIc}\right)}{\log_2 (\min \{ k,n \}+1)}\right\rceil$ for any $k \in \Z_+, k \geq 2$.
\end{enumerate}
\end{corollary}

Corollary \ref{cor:cor1} shows that if the integrality gap is high, then we cannot expect the split rank, the multi-branch split rank or the lattice-free rank of a well-behaved packing set to be low. 

To the best of our knowledge, the only other paper analyzing the rank of general lattice-free closures is \cite{averkov2017approximation}, and the only papers presenting lower bounds on the rank of multi-branch split closure for very special {\cred kinds} of polytopes are \cite{dash2013t} and \cite{li:ri:2008}. We note that none of these bounds are related to the {\cred integrality} gap.

There have been a number of papers giving lower bounds on split ranks such as \cite{BasuCM12, bodur2017cutting, cook:ka:sc:1990, dey:lowerbnd:2009, DeyL11} and bounds on a closely related concept, the {\cred reverse} split rank \cite{ConfortiPSFSpli15}. To the best of our knowledge, this is the first work connecting the integrality gap to the split rank. 
We note that the first part of Corollary \ref{cor:cor1} can be seen as a generalization of the result given in \cite{PokuttaS11} for the CG rank.

The lower bound on the split rank given in Corollary \ref{cor:cor1} is tight within a constant factor as formally stated below.

\begin{proposition}
\label{prop:prop1}
There exists a well-behaved packing polyhedron $Q$ and a nonnegative objective function $c$ such that $\rank_{\S}(Q) \le O\left(\textup{log}_2\left( \frac{\zLPc}{\zIc}\right)\right)$.
\end{proposition}

The proof of Proposition \ref{prop:prop1} is given in Section~\ref{subsubsec:4.5.4}.

\subsection{Covering}

We now state our results for covering {\cred polyhedra}. All the proofs regarding the covering case are given in Section~\ref{sec:Covering}.

\begin{theorem}
\label{thm:mainAppCovering}
Let $P_C$ be well-behaved. Then, the followings hold:
\begin{itemize}
\item[(i)] $\S(P_C)$ is a well-behaved covering polyhedron. 
\item[(ii)] $\frac{1}{2} P_C   \seq \S(P_C) $. 
\end{itemize} 
Moreover, the bound given in (ii) is tight, i.e., for every $\epsilon >0$, there exists a well-behaved covering polyhedron $\tilde{P}_C$ such that $\frac{1}{2-\epsilon} \tilde{P}_C  \not \seq \S(\tilde{P}_C)$.
\end{theorem}

Regarding part (i) of Theorem \ref{thm:mainAppCovering}, $\S(P_C)$ is known to be a rational polyhedron since $P_C$ is assumed to be a rational polyhedron~{\cred \cite{cook:ka:sc:1990}}, and it is straightforward to show that the split closure is of covering type (Proposition \ref{prop:CoveringPoly}); whereas its well-behavedness can be proven by showing that each split cut that violates the well-behavedness property is dominated by a well-behaved split cut (Proposition \ref{prop:CoveringSplitWell}). Proof of part (ii) follows from a case analysis that gives the correct factor of {\cred $\frac{1}{2}$} (Proposition \ref{prop:covering2approx}). For the last statement in the theorem, we provide a tight example in Proposition \ref{prop:LBAC}.

Note that a {\cred result similar} to Theorem \ref{thm:mainAppCovering} is proven in \cite{bodur2016aggregation} with respect to the {\cred aggregation closure}. However, Theorem \ref{thm:mainAppCovering} is not implied by the result of \cite{bodur2016aggregation} since for covering polyhedra, split cuts are not dominated by {\cred aggregation cuts}, see the example given in Observation \ref{obs:CoveringSplitAgg} in Appendix \ref{sec:appendix}.

Similar to the proof of Proposition \ref{thm:thm1} in the packing case, Theorem \ref{thm:mainAppCovering} yields the following lower bound on the split rank of covering polyhedra.

\begin{corollary}
\label{cor:SplitRankPolyCovering}
Let $P_C$ be well-behaved. Then, 
$$\rank_{\S}(P_C) \geq {\cred \sup_{c \in \R_+^n}} \left\lceil \textup{log}_2\left( \frac{\zIc}{\zLPc}\right)\right\rceil.$$
\end{corollary}

Unlike the packing case, we are unable to generalize the result of Corollary \ref{cor:SplitRankPolyCovering} for the case of $k$-lattice-free rank. The key technical argument that is a roadblock is to prove the well-behavedness of the $k$-lattice-free closure of covering polyhedron.  {\cred We do not know if the $k$-lattice-free closure of covering polyhedron is well-behaved. Note that in contrast} in the packing case, if we start from a well-behaved set and the closure is of packing type, then trivially the closure is also well-behaved.



\section{Proofs for packing problems}
\label{sec:Packing}
We use the following observation, from \cite{bodur2016aggregation}, in some of the proofs.
\begin{observation} \label{obs:bijection}
Let $\phi:\mathbb{R}^n \rightarrow \mathbb{R}^n$ be a bijective map, let $\{S^i\}_{i \in I}$ be a collection of subsets of $\mathbb{R}^n$ and {\cred for $S \seq \mathbb{R}^n$} let $\phi(S):= \{ \phi(x) \,|\, x \in S\}$. Then $\phi\left(\bigcap_{i \in I} S^i\right) = \bigcap_{i \in I} \phi(S^i)$.
\end{observation}
\subsection{Proof of Proposition \ref{thm:thm1}}
\label{subsec:4.1}
Let $\M$ be a collection of strict lattice-free sets which satisfies the packing invariance and
constant approximation properties. Let $Q \seq \R^n$ be a well-behaved packing set. Since $Q^I \seq \M(Q)$, we have that $e_j \in \M(Q)$ for all $j \in [n]$. Therefore, by the packing invariance property, $\M(Q)$ is also a well-behaved packing set.

{\cred Assume that the rank of the closure $\M$ is finite, as there is nothing to prove otherwise.} Let $t = \rank_{\M}(Q)$ and let $c \in \R_+^n$ be a given objective vector. Define $z^{i}$ to be the optimal objective function value of maximizing $c^\top x$ over the $i^{\text{th}}$ closure with respect to $\M$ of $Q$. Since, $\M(Q)$ is a well-behaved packing set, by induction, the $i^{\text{th}}$ closure with respect to $\M$ of $Q$ is a well-behaved packing set. Therefore, the constant approximation property guarantees that $z^i \le {\cred \alpha_\M} z^{i+1}$. Thus,
\begin{eqnarray*}
\frac{\zLPc}{\zIc} = \frac{\zLPc}{z^1} \frac{z^1}{z^2}\dots\frac{z^{t-1}}{z^{t}} \leq {(\alpha_\M)}^{t}.
\end{eqnarray*}
This implies the inequality
\begin{eqnarray*}
t = \rank_{\M}(Q) \geq \left\lceil \frac{\textup{log}_2\left( \frac{\zLPc}{\zIc}\right)}{\textup{log}_2 \alpha_\M} \right\rceil,
\end{eqnarray*}
which is the required result.
\subsection{Proof of Theorem \ref{thm:thm2}}
\label{subsec:4.2}
Let $\M$ be a collection of strict lattice-free sets with the following property. For $\Tset \seq [n]$, define $H[\Tset] := \{ x \in \R^n | x_j = 0, \ \forall j \in T\}$. Given $\Mset \in \M$, let 
$$M[\Tset] := ( M \cap H[\Tset] ) + \spann ( \{ e_{j} \}_{j \in T} ). $$ 
Assume that for any $\Mset \in \M$ and $T \seq [n]$, if $M[\Tset] \neq \emptyset$, then $M[\Tset] \in \M$. We will show that $\M$ is packing invariant. 

Let $Q$ be a packing set. If $Q$ is empty, then there is nothing to prove. Therefore, assume that $Q$ is nonempty.

Let $\Mset \in \M$. Let $\bpar^\top x \le \delta$ be a valid inequality for $Q^\Mset$. We will show that  this inequality is dominated by a packing type inequality valid for $\M(Q)$. Let $T = \{ j \in [n] : \bpar_j < 0\}$. If $T = \emptyset$, there is nothing to prove. So, assume that $T \neq \emptyset$. 


For convenience, we define an operator $\breve{(\cdot)}$ as follows: For a given vector $u \in \R^n$,  $\breve{u} \in \R^n$ is 
	\begin{align}
\breve{u}_j = \left \{
\begin{array}{rl}
u_j, & \text{if} \ j \in [n] \setminus T \\
0, & \text{if} \ j \in T
\end{array}
\right. .  \label{eq:opBreve}
	\end{align}

We will show that $\breve \bpar^\top x \le \delta$ is a valid inequality for $\M(Q)$.
Since $\breve \bpar \in \R^n_+$ and $\{x \in \R^n_+ : \breve \bpar x \le \delta \} \subseteq \{x \in \R^n_+ : \bpar^\top  x \le \delta \}$, we obtain the required result.

Let $\bar Q := Q \cap H[T]$.
As $\bar Q \subseteq Q$, we have that $\bpar^\top x \le \delta$ is a valid inequality for ${\bar Q}^\Mset$.  Since $\breve \bpar^\top x = \bpar^\top x$ for every $x \in H[T]$, we obtain that 
\begin{equation}
\label{betaineqValidForQbarM}
\breve \bpar^\top x \le \delta \ \text{is a valid inequality for}  \ {\bar Q}^\Mset. 
\end{equation}

Now, we {\cred distinguish} two cases:
\begin{itemize}
\item[Case 1.] $H[\Tset] \cap \Mset = \emptyset$: In this case, we know that $\bar Q = {\bar Q}^\Mset$, thus, using \eqref{betaineqValidForQbarM}, we have that $\breve \bpar x \leq \delta$ is valid for $\bar Q$. We show that $\breve \bpar^\top x \le \delta$ is valid for $Q$, and therefore trivially for $\M(Q)$. 
Assume by contradiction that there is a point $x \in Q$ such that $\breve \bpar^\top x > \delta$.
We have $\breve \bpar^\top \breve x = \breve \bpar^\top x > \delta$.
As $Q$ is a packing set, we have $\breve x \in Q$.
Moreover, since $\breve x \in H[T]$, we have $\breve x \in \bar Q$.
Thus $\breve x$ is a vector in $\bar Q$ with $\breve \bpar^\top \breve x > \delta$, a contradiction since $\breve \bpar x \leq \delta$ is valid for $\bar Q$.
Therefore, in this case the statement is trivially satisfied.

\item[Case 2.] $H[\Tset] \cap \Mset \neq \emptyset$: By the definition of $M[T]$, we have that $\Mset[\Tset] \neq \emptyset$ and 
$$H[\Tset] \cap \Mset = H[\Tset] \cap \Mset[T].$$
Therefore, $\bar Q \setminus \Mset = \bar Q \setminus \Mset[T]$, which together with \eqref{betaineqValidForQbarM} imply that
\begin{equation}
\label{factForContradictionMT}
\breve \bpar^\top x \le \delta \ \text{is a valid inequality for} \ {\bar Q}^{\M[T]}.
\end{equation}

We now show that $\breve \bpar^\top x \le \delta$ is a valid inequality for $Q^{\Mset[T]}$.
Assume by contradiction that there is a point $x \in Q \setminus \Mset[T]$ such that $\breve \bpar^\top x > \delta$.
We have $\breve \bpar^\top \breve x = \breve \bpar^\top  x > \delta$.
As $Q$ is a packing set, we have $\breve x \in Q$.
Moreover, since $\breve x \in H[T]$, we have $\breve x \in \bar Q$.
Finally, since $x \notin \Mset[T]$, we obtain that also $\breve x \notin \Mset[T]$ by definition of $\Mset[T]$.
Thus $\breve x$ is a vector in $\bar Q \setminus \Mset[T]$ with $\breve \bpar^\top \breve x > \delta$, a contradiction to \eqref{factForContradictionMT}.
\end{itemize}

\subsection{Proof of Theorem \ref{thm:thm3}}
\label{subsec:4.3}

Let $\M$ be a collection of strict lattice-free sets with the following property: There exist $\alpha_\M \geq 1$ such that $P_P \seq \alpha_\M \M(P_P)$ for every well-behaved packing polyhedron $P_P$. Let $Q$ be a well-behaved packing set. We will show that $Q \seq \alpha_\M \M(Q)$.

Our strategy to prove this statement is to first construct, in Lemma \ref{prop:SantanuLemma}, a well-behaved packing polyhedron which is an inner approximation of $Q$ and {\cred can be chosen} arbitrarily close to $Q$.. Then, we apply the $\alpha_\M$ factor to this polyhedral approximation and ``transfer" it to $Q$. 

\begin{lemma}
\label{prop:SantanuLemma} 
Let $\epsilon > 0$. Then, there exists a well-behaved packing polyhedron $P_\epsilon$ such that 
$\frac{1}{1+\epsilon} Q \subseteq P_\epsilon \subseteq Q$. 
\end{lemma}

\begin{proof}
{\cred First, consider the case that $Q$ is bounded.}  Let $\sigma_Q$ be the support function of $Q$, i.e.,
$$\sigma_Q (u) = \sup \{ u^\top x | x \in Q \},$$ 
and 
$$C^n := \{ u \in \R^n_+ \, | \, \|u\|_2 = 1 \}.$$
Also, let $\tilde{Q} = \frac{1}{1+\epsilon} Q$. We first show that there exists $M > 0$ such that 
\begin{equation}
\label{eq:sigmaLB}
\sigma_{\tilde{Q}}(u) \geq M \ \text{for all} \ u \in C^n.
\end{equation}
Let $S = \{ x \in \R^n_+ ~|~ \ones^\top x \leq 1 \}$ and $\tilde{S} = \frac{1}{1+\epsilon} S$. Since $Q$ is well-behaved, we have that $S \subseteq Q$, thus $\tilde{S} \subseteq \tilde{Q}$. Therefore, $\sigma_{\tilde{S}}(u) \leq \sigma_{\tilde{Q}}(u)$ for all $u \in C^n$. Since {\cred $\sigma_{\tilde{S}}(u) \geq \frac{1}{\sqrt{n}(1+\epsilon)}$} for all $u \in C^n$, \eqref{eq:sigmaLB} holds.

Let $\bar{M} = \max \{ \|x\|_\infty ~|~ x \in \tilde{Q} \}$. It is well-known that $\sigma_{\tilde{Q}} (\cdot)$ is continuous since $\tilde{Q}$ is a compact convex set \cite{rockafellar:1970}. Moreover, as $\|\cdot\|_2$ is also continuous, {\cred for any $u \in C^n$ and $\epsilon > 0$ there exists a neighborhood $N_u$ of $u$ (in the topology of the sphere)} such that for all $v \in N_u$ we have 
{\cred
\begin{equation}
\label{eq:contsigma}
|\sigma_{\tilde{Q}}(u) - \sigma_{\tilde{Q}}(v)| \le \frac{\epsilon M}{4}
\end{equation}
}
and
\begin{equation}
\label{eq:cont2norm}
\| u - v\|_2 \leq \frac{\epsilon M}{4 \bar{M}\sqrt{n}}.
\end{equation}
Since $C^n$ is a compact set, there exists a finite list of vectors $v_1,\hdots,v_\ell $ such that $C^n = \cup_{i=1}^\ell N_{v_i}$. Define 
\begin{equation}
\label{eq:Psetdef}
P^1_\epsilon := \{ x \in \R^n_+ ~|~ (v_i)^\top x \leq \sigma_{\tilde{Q}}(v_i), \forall i=1,\hdots,\ell, \ \text{and} \ x_i \leq \bar{M}, \forall i=1,\hdots,n \}
\end{equation}

We now show that 
\begin{equation}
\label{eq:innerapprox}
\tilde{Q} \subseteq P^1_\epsilon \subseteq (1+\frac{\epsilon}{2}) \tilde{Q} \subseteq Q.
\end{equation}
Note that the first and the last containments are straightforward. In order to show the second containment, we {\cred show} that $\sigma_{P^1_\epsilon}(u) / \sigma_{\tilde{Q}}(u) \leq 1+\frac{\epsilon}{2}$ for all $u \in C^n$. For a given $u \in C^n$, let $i \in \{ 1, \hdots,\ell\}$ such that $u \in N_{v_i}$. Observe that 
\begin{align}
\sigma_{P^1_\epsilon}(u) & \leq \sigma_{P^1_\epsilon}(v_i) + \sigma_{P^1_\epsilon}(u-v_i) \nonumber \\
& \leq \sigma_{P^1_\epsilon}(v_i) + \| u - v_i \|_2 \cdot \max_{x \in P^1_\epsilon} \{ \| x \|_2 \} \nonumber \\
& \leq \sigma_{P^1_\epsilon}(v_i) + \frac{\epsilon M}{{\cred 4} \sqrt{n} \bar{M}} \sqrt{n} \bar{M} \nonumber \\
& = \sigma_{P^1_\epsilon}(v_i) + \frac{\epsilon M}{4} \nonumber \\
& \leq \sigma_{\tilde{Q}}(v_i) + \frac{\epsilon M}{4} \nonumber \\
& \leq \sigma_{\tilde{Q}}(u) + \frac{\epsilon M}{4} + \frac{\epsilon M}{4} \nonumber \\
& = \sigma_{\tilde{Q}}(u) + \frac{\epsilon M}{2},  \label{eq:fifthineq} 
\end{align}
where the first inequality is due to the subadditivity property of the support functions \cite{rockafellar:1970}, the second is due to the Cauchy-Schwartz inequality, the third follows from \eqref{eq:cont2norm} and \eqref{eq:Psetdef}, the fourth inequality is implied by the constraints defining $P_\epsilon$ in \eqref{eq:Psetdef}, and the last inequality is satisfied by \eqref{eq:contsigma}. Inequality \eqref{eq:fifthineq} can be written as
$$\frac{\sigma_{P^1_\epsilon}(u)}{\sigma_{\tilde{Q}}(u)} \leq 1+ \frac{\epsilon M}{2 \sigma_{\tilde{Q}}(u)} \leq 1+ \frac{\epsilon}{2},$$
{\cred the second inequality} follows from  \eqref{eq:sigmaLB}.

Due to \eqref{eq:innerapprox}, $P^1_\epsilon$ achieves almost all the required conditions except the fact that it may not be well-behaved. Therefore, let 
$$P_\epsilon = \conv{(P^1_\epsilon \cup S)}.$$
First, note that $\tilde{Q} \subseteq P^1_\epsilon \subseteq P_\epsilon \subseteq Q$ where the first two containments are straightforward, and the last containment follows from the fact that $S \subseteq Q$ and $P^1_\epsilon \subseteq Q$. It remains to verify that $P_\epsilon$ is a packing polyhedron which would imply that it is {\cred a} well-behaved packing polyhedron since $S \subseteq P_\epsilon$. However, observe that $P_\epsilon$ is the convex hull of the union of two packing polyhedra, and therefore it is straightforward to verify that it is a packing polyhedron.

{\cred Now suppose $Q$ is not bounded. Then we can decompose it as $Q = B + R$, where $B$ is a bounded packing set and $R$ is the recession cone of $Q$; explicitly, let $I = \{i \mid \cone(e_i) \subseteq Q\}$, so $B = Q \cap \{x \mid x_i \le 0~\forall i \in I\}$ and $R = \cone(\{e_i\}_{i \in I})$. Furthermore, let $\bar{B} = \conv(B \cup \bigcup_{i \in I} e_i)$, so that $\bar{B}$ is a \emph{well-behaved} bounded packing set; notice $Q = B + R = \bar{B} + R$. 
	
	Applying the proof above to the bounded set $\bar{B}$, we obtain a well-behaved packing \emph{polyhedron} $P_\epsilon$ satisfying $\frac{1}{1+\epsilon} \bar{B} \subseteq P_\epsilon \subseteq \bar{B}$. Then $P_\epsilon + R$ is a well-behaved packing polyhedron (the polyhedrality follows from the fact $\bar{B}$ is a polytope and $R$ is finitely generated, see for example Theorem 3.13 of~\cite{ConCorZam14b}). Finally, since $\alpha P_\epsilon + R = \alpha (P_\epsilon + R)$ for all $\alpha > 0$, $$\frac{1}{1+\epsilon} (\bar{B} + R) = \frac{1}{1 + \epsilon} \bar{B} + R \subseteq P_\epsilon + R \subseteq \bar{B} + R.$$ Since $\bar{B} + R = Q$, $P_\epsilon + R$ is the desired polyhedral approximation. This concludes the proof. 
	}



\end{proof}

Noting that 
$$\M(Q) = \bigcap_{\Mset \in \M} Q^\Mset,$$
it is sufficient to prove that 
\begin{equation}
\label{eq:Pkplusone}
Q  \subseteq  (\alpha_\M) \,  Q^\Mset,
\end{equation}
{\cred for an arbitrary $M \in \M$} (see Observation \ref{obs:bijection}).

Let $\epsilon > 0$ and $P_\epsilon$ be the well-behaved packing polyhedron satisfying the conditions of Lemma \ref{prop:SantanuLemma}. Then, observe that 
\begin{equation}
\label{eq:generalContainment}
\frac{1}{1+\epsilon} Q \subseteq P_\epsilon \subseteq (\alpha_\M) (P_\epsilon)^\Mset \subseteq (\alpha_\M) \,  Q^\Mset,
\end{equation}
where the first and the last containments follow due to $\frac{1}{1+\epsilon} Q \subseteq P_\epsilon \subseteq Q$, whereas the second one holds by assumption and the fact that $P_\epsilon$ is well-behaved. 

Note that \eqref{eq:generalContainment} can be written as $Q \subseteq (1+\epsilon) (\alpha_\M) \,  Q^\Mset$. Since $\epsilon$ can be arbitrarily small, we obtain that $Q \subseteq (\alpha_\M) \,  Q^\Mset$.
\subsection{Proof of Theorem \ref{thm:thm4}}
\label{subsec:4.4}
Note that it is sufficient to prove the statement for $\M \in \{ \S^k, \L^k \}$ since $\S$ is a special case of $\S^k$. We will use Theorem \ref{thm:thm2} to prove this statement. That is, letting $T \seq [n]$, we will show that for every $\Mset \in \M$, we have $\Mset[\Tset] \in \M$ as well. {\cred Recall the operator $\breve{(\cdot)}$ from equation \eqref{eq:opBreve}.}

\begin{itemize}
\item[Case of $\S^k$:]  Consider an arbitrary element of $\S^k$ as
$$M = \bigcup_{i \in [k]} S(\pi^i,\pi^i_0).$$
Observe that 
$$M[\Tset] = \bigcup_{i \in [k]} S(\breve \pi^i,\pi^i_0).$$
If $\breve \pi^i = \zeros$, then $S(\breve \pi^i,\pi^i_0) = \emptyset$. Therefore, $M[T]$ is also a $k$-branch split set since $k$-branch split is defined to be the union of at most $k$ split sets.

\item[Case of $\L^k$:]  Consider an arbitrary element of $\L^k$ as
$$M = \{ x \in \R^n | (\pi^i)^\top x < \pi^i_0, \ i = 1,\hdots,k \} .$$
Observe that 
$$M[\Tset] = \{ x \in \R^n | (\breve \pi^i)^\top x < \pi^i_0, \ i = 1,\hdots,k \} .$$
If $\breve \pi^i = \zeros$, then either the inequality $(\breve \pi^i)^\top x < \pi^i_0$ is trivially satisfied, or $M[\Tset] = \emptyset$. Therefore, $M[T]$ is also a $k$-lattice-free set since $k$-lattice-free set is defined to be the union of at most $k$ lattice-free sets.

\end{itemize}
\subsection{Proof of Theorem \ref{thm:thm5}}

{\cred In order to make the proofs more self-contained, we record here standard bounds on the integrality gap of well-behaved packing polyhedra, which are essentially Proposition 6 of~\cite{bodur2016aggregation}.
	
	\begin{proposition} \label{prop:intGapPack}
		Consider a well-behaved packing polyhedron $P_P = \{x \in \R^n_+ \mid (a^i)^{\top} x \le b_i,~\forall i \in [m]\}$. Then $P_P$ is a $(\min\{m,n\} + 1)$-approximation of the integer hull $P_P^I$.
	\end{proposition}
	
	\begin{proof}
		It is equivalent to show that $P_P$ is both an $(m+1)$- and an $(n+1)$-approximation of $P_P^I$. The former is precisely Proposition 6 of~\cite{bodur2016aggregation}, and the latter follows from similar arguments, reproduced here. Given a cost vector $c \in \R^n_+$, we need to show that $\max\{c^\top x \mid x \in P_P\} \le (n + 1) \max\{c^\top x \mid x \in P_P^I\} =: (n+1) z^I$. Let $x^{LP}$ be an optimal solution for the left-hand side, and let $\hat{x}$ be the integer solution obtained by rounding down each component of $x^{LP}$. Since each component of the difference $x^{LP} - \hat{x}$ is in $[0,1]$, we have 
		\begin{align*}
			z^I \ge c^\top \hat{x} \ge c^{\top} x^{LP} - \|c\|_{\infty}.
		\end{align*}
		Moreover, since $P_P$ is well-behaved, all canonical vectors $e_i$ belong to $P_P^I$ and hence $z^I \ge \|c\|_{\infty}$. Adding $n$ times this lower bound to the displayed equation, we obtain $(n+1)z^I \ge c^{\top} x^{LP}$, thus proving the desired result. This concludes the proof. 
	\end{proof}

	}
	
\subsubsection{Case of $\S$}
\label{subsubsec:4.5.1}

We show that $\alpha_\S = 2$ in Proposition \ref{prop:packing2approx}. The proof of Proposition \ref{prop:packing2approx} involves a reduction to analyzing split closure of a packing polyhedron in $\R^2$, and a case analysis in $\R^2$ gives the correct factor of 2 (Lemma \ref{lem:pack_2approx}). For the last statement in the theorem, we provide a tight example in Proposition \ref{prop:pack_TightEx}.

{\cred We start with the proof of the result for the two-dimensional case.}
\begin{lemma}
\label{lem:pack_2approx}
Let $P_P \subseteq \R^2$ be well-behaved. Then $P_P \seq 2 \S(P_P)$.
\end{lemma}

\begin{proof}
{\cred
By \cite{cook:ka:sc:1990} and by Theorem \ref{thm:thm4}, the set $S(P_P)$ is a well-behaved packing polyhedron.
To show that $P_P \seq 2 \S(P_P)$, we just need to show that for all facet-defining inequalities $\bpar^\top x \le \delta$ of $S(P_P)$, the inequality $\bpar^\top x \le 2  \delta$ is valid for $P_P$.
This is trivially satisfied for the facet-defining inequalities of $S(P_P)$ of the type $x_i \ge 0$, thus it remains to be shown for the other facet-defining inequalities of $S(P_P)$.
Since $S(P_P)$ is of packing type, such facet-defining inequalities are of the form $\bpar^\top x \le \delta$ with $\bpar \in \R^2_+$.
Since the split closure is finitely generated \cite{Ave12}, each facet-defining inequality $\bpar^\top x \le \delta$ of $S(P_P)$ defines a facet of a set $P_P^T := \conv(P_P \setminus \int T)$, where $T$ is a split set.
(Note that the polyhedra $P_P^T$ are not necessarily of packing type.)
Therefore, to complete the proof of the lemma, it suffices to show that for every split set $T$, and for every inequality $\bpar^\top x \le \delta$ with $\bpar \in \R^2_+$  valid for $P_P^T$, the inequality $\bpar^\top x \le 2  \delta$ is valid for $P_P$.
We show that
for every split set $T$, and for every $\bpar \in \R^2_+$, there exists $\hat x \in P_P^T$ that satisfies $\max\{\bpar^\top x \mid x \in P_P\} \le 2\bpar^\top \hat x$.
This completes the proof because $\bpar^\top \hat x \le \delta$ implies that every point in $P_P$ satisfies $\bpar^\top x \le \max\{\bpar^\top x \mid x \in P_P\} \le 2\bpar^\top \hat x \le 2 \delta$.}

{\cred Now, fix a split set $T$ and a vector} $\bpar \in \R^2_+$, and let $\bar x$ be a vector in $P_P$ that achieves $\max\{\bpar^\top x \mid x \in P_P\}$.
Since $P_P$ is a packing polyhedron, we have $\bar x \ge 0$.
We divide the proof in three main cases based on the position of vector $\bar x$.

1. In the first case we assume that $\bar x \ge (1,1)$, and we define $\hat x := \lfloor \bar x \rfloor$.
Since $P_P$ is a packing polyhedron, we have that $\hat x \in P_P \cap \Z^2 \subseteq P_P^T$.
As $\bar x \ge (1,1)$, we have $2 \hat x \ge \bar x$.
Finally, $\bpar \ge 0$ implies $2 \bpar^\top \hat x \ge \bpar^\top \bar x$ as desired.

2. In the second case we assume that $\bar x \le (1,1)$.
Since $P_P$ is well-behaved, we have that points $(1,0)$ and $(0,1)$ are in $P_P$ and therefore in $P_P^T$.
If $\bpar_1 \ge \bpar_2$, we define $\hat x := (1,0)$.
Then $2\bpar^\top \hat x = 2 \bpar_1 \ge \bpar_1 + \bpar_2$.
Since $\bpar \ge 0$ and $\bar x \le (1,1)$, we have $\bpar_1 + \bpar_2 \ge \bpar^\top \bar x$, which implies $2\bpar^\top \hat x \ge \bpar^\top \bar x$ as desired.
Symmetrically, if $\bpar_2 \ge \bpar_1$, we define $\hat x := (0,1)$, and obtain $2\bpar^\top \hat x \ge \bpar^\top \bar x$.

3. In the third case we assume that $\bar x_1 < 1$ and $\bar x_2 > 1$.
(The same argument works for the symmetric case $\bar x_2 < 1$ and $\bar x_1 > 1$.)

Assume first that $T$ is \emph{not} a vertical split set {\cred $\{x \mid t \le x_1 \le t+1\}$} for some integer $t$.
Define now $\hat x^1 := \lfloor \bar x \rfloor = (0,\floor{\bar x_2})$.
Since $P_P$ is {\cred a packing polyhedron}, {\cred the} vector $\hat x^1$ is in $P_P \cap \Z^2$, and therefore in $P_P^T$.
If $2\bpar^\top \hat x^1 = 2 \bpar^\top (0,\floor{\bar x_2}) \ge \bpar^\top \bar x$, then we are done, thus we now assume $2 \bpar^\top (0,\floor{\bar x_2}) \le \bpar^\top \bar x$.

Let $\hat x^2 := (\bar x_1, \bar x_2-1)$.
It can be shown that, since $T$ is not a vertical split set, the vector $\hat x^2$ is in $P_P^T$.
We show 
$2 \bpar^\top \hat x^2 = 2\bpar^\top (\bar x_1,\bar x_2 -1) \ge \bpar^\top \bar x$.
Since $\floor{\bar x_2} \ge 1$ and $\bpar_2 \ge 0$, we have 
$\bpar^\top \bar x \ge 2 \bpar_2 \floor{\bar x_2} \ge 2 \bpar_2$.
By adding $\bpar^\top \bar x$ to both sides we obtain
$2 \bpar^\top \bar x - 2 \bpar_2 \ge \bpar^\top \bar x$,
thus
$2\bpar^\top (\bar x_1,\bar x_2 -1) \ge \bpar^\top \bar x$.

Finally, assume that $T$ is a vertical split set {\cred $\{x \mid t \le x_1 \le t+1\}$} for some integer $t$.
Define now $\hat x^1 := (1,0)$.
Since $P_P$ is well-behaved, {\cred the} vector $\hat x^1$ is in $P_P \cap \Z^2$, and therefore in $P_P^T$.
If $2\bpar^\top \hat x^1 = 2 \bpar^\top (1,0) \ge \bpar^\top \bar x$, then we are done, thus we now assume $2 \bpar^\top (1,0) \le \bpar^\top \bar x$.
Define $\hat x^2 := (0, \bar x_2)$ and note that $\hat x^2 \in P_P^T$ since $T$ is a vertical split set.
We show 
$2\bpar^\top\hat x^2 = 2\bpar^\top (0, \bar x_2) \ge \bpar^\top \bar x$.
Since $\bpar_1 \ge 0$ and $\bar x_1 < 1$, we have $2 \bpar_1 \bar x_1 < 2 \bpar_1$.
By summing the latter with $2 \bpar_1 \le \bpar^\top \bar x$ we obtain $2 \bpar_1 \bar x_1 \le \bpar^\top \bar x$.
By adding and subtracting $2\bpar_2 \bar x_2$ to the left-hand, we get $2 \bpar^\top \bar x - 2\bpar_2 \bar x_2 \le \bpar^\top \bar x$ which implies $2 \bpar_2 \bar x_2 \ge \bpar^\top \bar x$ as desired.
\end{proof}
\begin{proposition}[$\alpha_S = 2$]
\label{prop:packing2approx}
Let  $Q \seq \R^n$ be a well-behaved packing set. Then, 
$Q \seq 2 \S(Q).$
\end{proposition}
\begin{proof}
It is sufficient to prove this proposition for a packing polyhedron, $P_P$, due to Theorem \ref{thm:thm3}.
Let $(\pi,\pi_0) \in \Z^n \times \Z$ and let $\bpar^\top x \leq \delta$ be a valid inequality for $(P_P)^{\pi,\pi_0}$. Note that, due to Observation \ref{obs:bijection}, it is sufficient to show that $\bpar^\top x \leq 2 \delta$ is valid for $P_P$. Due to Farkas' Lemma {\cred (e.g., Theorem 3.22 in~\cite{ConCorZam14b})}, there exist ${\cred \lambda^1, \lambda^2} \in \R_+^m, \ {\cred \mu_1, \mu_2 \in \R_+}$ and ${\cred \sigma^1, \sigma^2} \in \R_+^n$ such that for any $j \in [n]$, we have
$$\bpar_j  = \sum_{i=1}^m \lambda^1_i A_{ij} + \mu_1 \pi_j - \sigma^1_j = \sum_{i=1}^m \lambda^2_i A_{ij} - \mu_2 \pi_j - \sigma^2_j.$$
Let
$$Q := \{ x \mid (\lambda^1)^\top Ax \leq (\lambda^1)^\top b, ~ (\lambda^2)^\top Ax \leq (\lambda^2)^\top b,  ~ x \geq 0\}.$$
Now, observe that $Q \supseteq P_P$. Therefore, it is sufficient to show that $\bpar^\top x \leq 2 \delta$ is valid for $Q$. We will prove that the following holds:
\begin{equation}
\label{eq:QclaimPacking}
Q \subseteq 2 Q^{\pi,\pi_0}.
\end{equation}
Since $\bpar^\top x \leq \delta$ is valid for $Q^{\pi,\pi_0}$ by the definition of $Q$, this will imply that  $\bpar^\top x \leq 2 \delta$ is valid for $Q$.

In order to show that \eqref{eq:QclaimPacking} holds, we verify that 
\begin{equation}
\label{eq:minPacking}
\max \{ c^\top x \mid x \in Q\} \leq 2 \max \{ c^\top x \mid x \in Q^{\pi,\pi_0} \},
\end{equation}
for any objective vector $c \in \R_+^n$. Let $x^*$ be a vertex of $Q$ that maximizes $c^\top x$ over $Q$. As $Q$ is defined by two linear inequalities, together with non-negativities, we know that at least $n-2$ components of $x^*$ are zero, say $x^*_j = 0$ for all $j=3,\hdots,n$. We will focus on the restriction of $Q$ to the first two variables, which we denote by $Q |_{\R^2}$. 

Observe that 
\begin{equation}
\label{eq:qc1}
\max \{ c^\top x \mid x \in Q\} = \max \{ c_1 x_1 + c_2 x_2 | (x_1,x_2) \in Q |_{\R^2} \}.
\end{equation}
Moreover, we have
\begin{equation}
\label{eq:qc2}
\max \{ c^\top x \mid x \in Q^{\pi,\pi_0} \} \geq \max \{ c_1 x_1 + c_2 x_2 | (x_1,x_2) \in (Q |_{\R^2})^{\pi,\pi_0} \} 
\end{equation}
because {\cred $(Q |_{\R^2})^{\pi,\pi_0} \subseteq Q^{\pi,\pi_0} |_{\R^2}$.} 

Due to \eqref{eq:qc1} and \eqref{eq:qc2}, in order to prove \eqref{eq:minPacking}, it is sufficient to only prove \eqref{eq:minPacking} in $\R^2$. Since $Q |_{\R^2}$ is well-behaved, this immediately follows from Lemma \ref{lem:pack_2approx}.
\end{proof}

\begin{proposition}[Tight example]
\label{prop:pack_TightEx}
For every $\epsilon >0$, there exists a well-behaved packing polyhedron $\tilde{P}_P$ such that $\tilde{P}_P \not\subseteq (2-\epsilon) \S(\tilde{P}_P)$. 
\end{proposition}

\begin{proof}
Let $\epsilon > 0$ and $M = \max\{1,\lceil \frac{2}{\epsilon}-1 \rceil \}$.
Consider the instance 
$\textup{max} \{ x_1 + x_2 \,|\, x \in \tilde{P}_P\}$, where 
\begin{align*}
\tilde{P}_P = \{x \in \R^2_+ \,|\, x_1 + Mx_2 \leq M, \ Mx_1 + x_2 \leq M \}.
\end{align*}
Note that $\tilde{P}_P$ is well-behaved.
It is sufficient to show that 
$\frac{z^{LP}}{z^{\S}} \ge 2 - \epsilon$ 
for this instance. 

\begin{enumerate}
\item 
$z^{LP} \ge \frac{2M}{M + 1}$: 
It can be checked that the point $\bar x_1 = \bar x_2 = \frac{M}{M + 1}$ is in $\tilde{P}_P$. 
Thus, 
$z^{LP} \ge \frac{2M}{M + 1}$. 
\item 
$z^{\S} \le1$: 
Adding the two constraints defining $\tilde{P}_P$ we obtain the valid inequality
\begin{eqnarray*} 
x_1 + x_2 \le \frac{2M}{M + 1}
\end{eqnarray*}  
The corresponding CG cut is $x_1 + x_2 \leq 1$. Since each CG cut is also a split cut we obtain $z^{\S} \le 1$.
\end{enumerate}
Thus, 
$\frac{z^{LP}}{z^{\S}} \ge \frac{2M}{M + 1}$; and our choice of $M$ completes the proof. 
\end{proof}

We note that the example given in Proposition \ref{prop:pack_TightEx} is the same as the one used in \cite{bodur2016aggregation} to show that the $2$-approximation bound for the CG closure of a well-behaved packing polyhedron is tight.

\subsubsection{Case of $\S^k$}
\label{subsubsec:4.5.2}

We will show that {\cred we can choose $\alpha_{\S^k} = \min \{ 2^k,n \}+1$}. It is sufficient to prove this proposition for a packing polyhedron, $P_P$, due to Theorem \ref{thm:thm3}.

Let $P_P = \{ x \in \R^n | Ax \leq b, x \geq 0\}$ and $\pi^i \in \Z^n, \pi^i_0 \in \Z \textup{ for all } i \in [k]$. It is sufficient to prove that $(\min\{2^k,n\}+1) \, (P_P)^{\pi^1,\dots, \pi^k;\pi^1_0, \dots, \pi^k_0} \supseteq  P_P$.

Let $\bpar^\top x \leq \delta$ be a valid inequality for $(P_P)^{\pi^1,\dots, \pi^k;\pi^1_0, \dots, \pi^k_0}$. Since \\ $\zeros \in (P_P)^{\pi^1,\dots, \pi^k;\pi^1_0, \dots, \pi^k_0}$, we have $\delta \geq 0$. Therefore, it is sufficient to prove that 
$$(\min\{2^k,n\}+1) \left(  \{ x | \bpar^\top x \leq \delta \} \right) \supseteq  P_P.$$
Let $\mathcal{G} = \{ G \seq {\cred [k]} : (P_P)_G^{{\pi}^1,\dots, {\pi}^k;\pi^1_0, \dots, \pi^k_0} \neq \emptyset \}$, where $(P_P)_G^{{\pi}^1,\dots, {\pi}^k;\pi^1_0, \dots, \pi^k_0}$ is defined as 
$$(P_P)_G^{\pi^1,\dots, \pi^k;\pi^1_0, \dots, \pi^k_0} = P_P \medcap \left(\bigcap_{i \in G} \{ (\pi^i)^\top x \geq \pi^i_0 + 1)\} \right)\medcap \left(\bigcap_{i \in {\cred [k]} \setminus G} \{ (\pi^i)^\top x \leq \pi^i_0)\}\right).$$

By Farkas' Lemma, we know that $\bpar^\top x \leq \delta$ is valid for 
\begin{equation}
\label{eq:aggPPG}
\{ x \in \R_+^n | (\lambda^G)^\top A x \leq (\lambda^G)^\top b, (\pi^i)^\top x \geq \pi^i_0 + 1, \ \forall \ i \in G, \ (\pi^i)^\top x \leq \pi^i_0, \ \forall \ i \in {\cred [k]} \setminus G \},
\end{equation}
for some $\lambda^G \in \R_+^m$.

Let 
$$Q = \{ x \in \R_+^n | (\lambda^G)^\top A x \leq (\lambda^G)^\top b, \ \forall G \in \mathcal{G} \} $$
which is well-behaved since $P_P$ is assumed to be well-behaved. Now, observe that 
\begin{align*}
(\min\{2^k,n\}+1) \left(  \{ x | \bpar^\top x \leq \delta \} \right) & \supseteq (\min\{|\mathcal{G}|,n\}+1) \left(  \{ x | \bpar^\top x \leq \delta \} \right)  \\
& \supseteq (\min\{|\mathcal{G}|,n\}+1) Q^I \supseteq Q \supseteq P_P,
\end{align*}
where the second containment follows from \eqref{eq:aggPPG}, the third one follows from {\cred Proposition \ref{prop:intGapPack}} since $Q$ is well-behaved, and the last one is straightforward.

\subsubsection{Case of $\L^k$}
\label{subsubsec:4.5.3}

We show that {\cred we can choose $\alpha_{\L^k} = \min \{ k,n \}+1$}. It is sufficient to prove this proposition for a packing polyhedron, $P_P$, due to Theorem \ref{thm:thm3}.

Let $P_P = \{ x \in \R^n | Ax \leq b, x \geq 0\}$ and {\cred let}
$$L = \{ x \in \R^n | (\pi^j)^\top x \leq \pi^j_0, \ j = 1,\hdots,k \}$$
{\cred be lattice-free.}
Then, observe that 
\begin{equation}
\label{eq:convPPL}
(P_P)^L = \conv (P_P \setminus \int(L)) = \conv  \left(\bigcup_{j=1}^k \left\{ x \in P_P ~|~ (\pi^j)^\top x \geq \pi^j_0 \right\}  \right).
\end{equation}
Without loss of generality, assume that the set $\{ x \in P_P ~|~ (\pi^j)^\top x \geq \pi^j_0 \}$ is non-empty if $j \leq r$, and empty otherwise, for some $r$ with $1 \leq r \leq k$.

Let $\bpar^\top x \leq \delta$ be a valid inequality for $(P_P)^L$. Since the origin is contained in $(P_P)^L$, we have $\delta \geq 0$. Therefore, it is sufficient to prove that 
$$(\min\{k,n\}+1) \left(  \{ x | \bpar^\top x \leq \delta \} \right) \supseteq  P_P.$$
By equation \eqref{eq:convPPL} and Farkas' Lemma, we know that $\bpar^\top x \leq \delta$ is valid for 
\begin{equation}
\label{eq:aggPP}
\{ x \in \R_+^n | (\lambda^j)^\top A x \leq (\lambda^j)^\top b, (\pi^j)^\top x \geq \pi^j_0 \},
\end{equation}
for $j = 1,\hdots,r$ where $\lambda^j \in \R_+^m$.

Let 
$$Q = \{ x \in \R_+^n | (\lambda^j)^\top A x \leq (\lambda^j)^\top b, \ j = 1,\hdots,r \} $$
which is well-behaved since $P_P$ is assumed to be well-behaved. Now, observe that 
$$(\min\{k,n\}+1) \left(  \{ x | \bpar^\top x \leq \delta \} \right) \supseteq (\min\{k,n\}+1) Q^I \supseteq Q \supseteq P_P,$$
where the first containment follows from \eqref{eq:aggPP} and $L$ being a lattice-free set, the second one follows from {\cred Proposition \ref{prop:intGapPack}} since $Q$ is well-behaved, and the last one is straightforward.
\subsubsection{Proof of Proposition \ref{prop:prop1}}
\label{subsubsec:4.5.4}
Let $P_P$ be the standard relaxation of the stable set polytope:
\begin{align*}
P_P = \{ x \in \R_+^{n} \,|\, x_i + x_j \le 1 \ \forall i,j \in [n], \ i < j\}.
\end{align*}
Corresponding to the clique inequality $\ones^\top x \le 1$, we optimize the all ones vector over $P_P$ and $(P_P)^I$, and obtain $z^{LP}=n/2$ and  $z^{I}=1$, respectively. 
The CG rank of the clique inequality is known to be $\lceil \log_2 (n-1) \rceil$ \cite{hartmann}, therefore it also constitutes an upper bound on the split rank.

\section{Proofs for covering problems}
\label{sec:Covering}

\begin{proposition} 
\label{prop:CoveringPoly}
$\S(P_C)$ is a covering polyhedron. 
\end{proposition}	
\begin{proof}
If $P_C$ is empty, then there is nothing to prove. So, assume that $P_C$ is not empty. It is known that the split closure of a polyhedron is also a polyhedron \cite{cook:ka:sc:1990}. Let $\bpar^\top x \geq \delta$ be a valid inequality for $\S(P_C)$. Then, there exists $(\pi,\pi_0) \in \Z^n \times \Z$ such that $\bpar^\top x \geq \delta$ is valid for $P_C^{\pi,\pi_0}$. If one side of the disjunction is empty, then we already know that $\bpar^\top x \geq \delta$ is of covering type as it is a CG cut. Now, assume that both sides are nonempty. Then, due to Farkas' Lemma, there exist multipliers $\lambda^1, \lambda^2 \in \R_+^m, \ \mu_1, \mu_2 \in \R_+$ and $\sigma^1, \sigma^2 \in \R_+^n$ 
for the aggregation
\begin{alignat*}{4}
& (\lambda^1) \quad  && Ax \geq b \qquad \qquad \qquad && (\lambda^2) \quad && Ax \geq b \\
& (\mu_1) && - \pi^\top x \geq -\pi_0 && (\mu_2) && \pi^\top x \geq \pi_0 +1 \\
& (\sigma^1_j) && x_j \geq 0 && (\sigma^2_j) && x_j \geq 0 \qquad \qquad j=1,\hdots,n
\end{alignat*}
such that, for any $j=1,\hdots,n$, 
\begin{equation}
\label{eq:cSplit}
\bpar_j  = \sum_{i=1}^m \lambda^1_i A_{ij} - \mu_1 \pi_j + \sigma^1_j = \sum_{i=1}^m \lambda^2_i A_{ij} + \mu_2 \pi_j + \sigma^2_j.
\end{equation}
This implies that, for any $j=1,\hdots,n$, we have $\bpar_j \geq 0$ (based on the sign of $\pi_j$, either the middle or the last expression {\cred witnesses non-negativity}). Lastly, note that if $\delta < 0$, then $\bpar^\top x \geq \delta$ is dominated by $\bpar^\top x \geq 0$, which concludes the proof.
\end{proof}

\begin{proposition}
\label{prop:covering2approx}
Let $P_C$ be well-behaved, i.e., $A_{ij} \leq b_i$ for all $i \in [m], j \in [n]$. Then, 
$$z^{LP} \geq \frac{1}{2} z^{\S}.$$ 
\end{proposition}
\begin{proof}
Let $\bpar^\top x \geq \delta$ be a facet-defining inequality for $\S(P_C)$. Note that, due to Observation \ref{obs:bijection}, it is sufficient to show 
that $\bpar^\top x \geq \frac{\delta}{2}$ is valid for $P_C$. {\cred If $\bpar^\top x \geq \delta$ is valid for $P_C$, then there is nothing to show. So, assume that it is a non-trivial inequality.}

If $\bpar^\top x \geq \delta$ is a {\cred non-trivial} CG cut, then {\cred $\delta \geq 1$. We know that the strict inequality $\bpar^\top x > \delta - 1$ is valid for $P_C$. If $\delta \geq 2$, then $\delta - 1 \geq \frac{\delta}{2}$, which implies that  $\bpar^\top x \geq \frac{\delta}{2}$ is valid for $P_C$. So, now assume that $\delta = 1$. Let $\lambda \in \R_+^m$ such that the CG cut is obtained by rounding up the coefficients and the right-hand-side of the base inequality $\lambda^\top A x \geq \lambda^\top b$, i.e., $\beta = \ceil{\lambda^\top A}$ and $\delta = \ceil{\lambda^\top b}$. As $\delta = 1$, we have $0 < \lambda^\top b \leq 1$. Thus, the scaled base inequality $(\lambda^\top A / \lambda^\top b) x \geq 1$ is valid for $P_C$. In addition, since $P_C$ is well-behaved, $\lambda^\top A_{\cdot j} \leq \lambda^\top b \ (\leq 1)$ for all $j \in [n]$ (where $A_{\cdot j}$ denotes the $j^\text{th}$ column of $A$. Then, for any $x \in P_C$,  we have 
$$\beta^\top x = \sum_{j \in [n]} \ceil{\lambda^\top A_{\cdot j}} x_j \geq \sum_{j \in [n]} \frac{\lambda^\top A_{\cdot j}}{\lambda^\top b} x_j \geq 1 \geq \frac{\ceil{\lambda^\top b}}{2} = \frac{\delta}{2},$$
which implies that is $\beta^\top x \geq \frac{\delta}{2}$ valid for $P_C$.
}

Otherwise, let $(\pi,\pi_0) \in \Z^n \times \Z$ be a corresponding vector such that $\bpar^\top x \geq \delta$ is valid for $P_C^{\pi,\pi_0}$. 
 Let
$$Q := \{ x \mid (\lambda^1)^\top Ax \geq (\lambda^1)^\top b, ~ (\lambda^2)^\top Ax \geq (\lambda^2)^\top b,  ~ x \geq 0\},$$
where $\lambda^1$ and $\lambda^2$ are the multipliers that satisfy \eqref{eq:cSplit}. Now, observe that $Q \supseteq P_C$. Therefore, it is sufficient to show that $\bpar^\top x \geq \frac{\delta}{2}$ is valid for $Q$. We will prove that the following holds:
\begin{equation}
\label{eq:Qclaim}
Q \subseteq \frac{1}{2} Q^{\pi,\pi_0}.
\end{equation}
Since $\bpar^\top x \geq \delta$ is valid for $Q^{\pi,\pi_0}$ by the definition of $Q$, this will imply that  $\bpar^\top x \geq \frac{\delta}{2}$ is valid for $Q$.

In order to show that \eqref{eq:Qclaim} holds, we verify that 
\begin{equation}
\label{eq:min}
\min \{ c^\top x \mid x \in Q\} \geq \frac{1}{2} \min \{ c^\top x \mid x \in Q^{\pi,\pi_0} \},
\end{equation}
for any objective vector $c \in \R_+^n$. Let $x^*$ be a vertex of $Q$ that minimizes $c^\top x$ over $Q$. If $x^*$ belongs to $Q^{\pi,\pi_0}$, we are done. Thus, assume that $x^* \notin Q^{\pi,\pi_0}$. We will prove \eqref{eq:min} by showing that there exists a point $\hat{x} \in Q^{\pi,\pi_0}$ such that $c^\top \hat{x} \leq 2 c^\top x^*$. As $Q$ is defined by two linear inequalities, together with non-negativities, we know that at least $n-2$ components of $x^*$ are zero, say $x^*_j = 0$ for all $j=3,\hdots,n$. We will focus on this restriction of $Q$ in $\R_+^2$ in order to identify $\hat{x}$. Without loss of generality, assume that $c_1 \geq c_2$. A key observation that follows from the definition of split cuts is
\begin{equation}
\label{eq:vee}
(x^*_1+1,x^*_2,\zeros) \in Q^{\pi,\pi_0} \vee (x^*_1,x^*_2+1,\zeros) \in Q^{\pi,\pi_0}.
\end{equation} {\cred Moreover, if $\pi \neq e_1$, then $(x^*_1,x^*_2+1,\zeros) \in Q^{\pi,\pi_0}.$\\}
Now, we will consider two cases to prove \eqref{eq:min}.
\\ \\
\emph{Case 1}. $x^*_1 \geq 1$:  
{\cred Using (\ref{eq:vee}), there are two subcases based on whether $(x^*_1+1,x^*_2,\zeros) \in Q^{\pi,\pi_0} $ or $ (x^*_1,x^*_2+1,\zeros) \in Q^{\pi,\pi_0}$. If $\hat{x} = (x^*_1+1,x^*_2,\zeros) \in Q^{\pi,\pi_0}$, then it is sufficient to show that 
$$c^\top x^* \geq \frac{1}{2} (c^\top x^* + c_1),$$
which is equivalent to $c^\top x^* \geq  c_1$, which holds because $x^*_1 \geq 1$ and $c_1, c_2, x^*_2 \geq 0$. If $\hat{x} = (x^*_1,x^*_2 + 1,\zeros) \in Q^{\pi,\pi_0}$, then it is sufficient to show that 
$$c^\top x^* \geq \frac{1}{2} (c^\top x^* + c_2),$$
which is equivalent to $c^\top x^* \geq  c_2$, which holds because $x^*_1 \geq 1$, $c_1 \geq c_2$ and $c_2, x^*_2 \geq 0$. }
 \\ \\
\emph{Case 2}. $0 \leq x^*_1 < 1$: Note that by construction, $Q$ is a well-behaved covering polyhedron. Now, consider the following two subcases:
 \\ \\	
\emph{Case 2a}. $(\pi,\pi_0) \neq (e_1,0)$: {\cred In this case as discussed above}, {\cred $\hat{x} = (x^*_1,x^*_2 + 1,\zeros) \in Q^{\pi,\pi_0}$}. It is sufficient to show that 
$$c^\top x^* \geq \frac{1}{2} (c^\top x^* + c_2),$$ 
which is equivalent to $c^\top x^* \geq c_2$.
This holds as we have $x^*_1+x^*_2 \geq 1$ since $Q$ is well-behaved {\cred and because $c_1 \geq c_2$}.
 \\ \\
\emph{Case 2b}. $(\pi,\pi_0) = (e_1,0)$: Let $\hat{x} = (x^*_1,x^*_2 + x^*_1 x^*_2,\zeros)$. We will first show that $\hat{x} \in Q^{\pi,\pi_0}$.

Figure \ref{fig:1pic} illustrates the restriction of $Q$ to the first two variables, which we denote by $Q |_{\R^2}$. Observe that $Q |_{\R^2}$ is a well-behaved covering polyhedron 
because $Q$ is a well-behaved covering polyhedron.  Note that $x_1^* > 0$ due to the assumption $x^* \notin Q^{e_1,0}$.
{\cred Since $0 < x^*_1 < 1$, and $Q |_{\R^2}$ is a well-behaved covering polyhedron, we  have that $x_1 \geq  \gamma$ cannot be a valid inequality for $Q |_{\R^2}$ where $0< \gamma \leq x^*_1$. In other words, all non-trivial inequalities $\alpha_1 x_1 + \alpha_2 x_2 \geq \theta$ defining $Q |_{\R^2}$ must have $\alpha_2 > 0$.  Therefore, there must exist a vertex of the form $(0,y)$ of $Q |_{\R^2}$.}

Since $(x^*_1,x^*_2)$ is a vertex of $Q |_{\R^2}$ {\cred with $x_1^* > 0$}, we know that $(0,y) \neq (x^*_1,x^*_2)$. 

Let $(h,0)$ be the intercept of the line passing through the vertices $(0,y) $ and $(x^*_1,x^*_2)$.  {\cred We first claim that $h \geq 1$. Note that the supporting hyperplane corresponding to non-trivial facet-defining inequality at $(0, y)$ (i.e., different from $x_1 \geq 0$) intersects the $x_1$-axis at a point $(\tilde{h}, 0)$ with $\tilde{h} \geq 1$ due to well-behavedness of $Q |_{\R^2}$. Since the line passing through the vertices $(0,y) $ and $(x^*_1,x^*_2)$ has an intercept at least as large as $\tilde{h}$, we obtain that $h \geq 1$.}

Then, we have
\begin{figure}[h]
\begin{center}
\begin{tikzpicture}[scale=1] 
    \draw[->] (0,0) -- (0,4.5);
    \draw[->] (0,0) -- (4.5,0);
    \draw (0,3.5) -- (1.5,1.5);
    \draw (1.5,1.5) -- (4,0);
    \draw[fill=gray,opacity=0.2] (0,4.5) -- (0,3.5) -- (1.5,1.5) -- (4,0) -- (4.5,0) -- (4.5,4.5) -- cycle;
    \draw[dotted, thick] (1.5,1.5) -- (2.635,0);
    \node[fill,circle,inner sep=1.1pt] at (4,0) {};
    \node[fill,circle,inner sep=1.1pt] at (1.5,1.5) {};
    \node (xstar) at (2.1,1.7) {\small $(x^*_1,x^*_2)$};
    \node[fill,circle,inner sep=1.1pt] at (0,3.5) {};
    \node (xstar) at (0.58,3.65) {\small $(0,y)$}; 
     \node (h) at (2.635,-0.3) {\small $(h,0)$}; 
     \node[inner sep=1.1pt] at (2.635,0) {\tiny x};
\end{tikzpicture}
\end{center}
\caption{The set $Q |_{\R^2}$}
\label{fig:1pic}
\end{figure}

\begin{equation}
\label{eq:heq}
\frac{y}{h} = \frac{y - x^*_2}{x^*_1} \Rightarrow h = \frac{y x^*_1}{y-x^*_2} \ .
\end{equation}
{\cred Since $h \geq 1$, we have:}
$$y x^*_1 \geq y - x^*_2 \iff \displaystyle y \leq \frac{x^*_2}{1-x^*_1} 
\Rightarrow \displaystyle (0, \frac{x^*_2}{1-x^*_1},\zeros) \in Q^{e_1,0}.$$
The last implication follows from the fact that $Q$ is a covering polyhedron, $(0,y,\zeros) \in Q$ and $(\pi,\pi_0) = (e_1,0)$.
Similarly, we have $(1,x^*_2,\zeros) \in Q^{e_1,0}$. The following convex combination of these two points yields $\hat{x}$ as
$$(1-x^*_1) (0, \frac{x^*_2}{1-x^*_1},\zeros) + x^*_1 (1,x^*_2,\zeros) = (x^*_1,x^*_2 + x^*_1 x^*_2,\zeros) = \hat{x}.$$
Finally, observe that {\cred since $c_1 \geq 0$ and $0 \leq x^*_1 \leq 1$}
$${\cred c^\top x^* \geq c_2 x^*_2 \geq  c_2 x^*_1 x^*_2 \Rightarrow c^\top x^* \geq \frac{1}{2} (c^\top x^* + c_2 x^*_1 x^*_2) = \frac{1}{2}c^\top\hat{x}}, $$
which completes the proof.
\end{proof}
We now show that Proposition \ref{prop:covering2approx} is tight.
In order to do so, we exhibit an instance of a well-behaved covering polyhedron and a nonnegative objective function such that LP is not better than a 2-approximation of $\S$. The construction given in the following example is the same that we used in \cite{bodur2016aggregation}
to show that our $2$-approximation bounds for \emph{1-row closure} and \emph{1-row CG closure} are tight.

\begin{proposition} 
\label{prop:LBAC}
For every $\epsilon >0$, there exists a well-behaved covering polyhedron $\tilde{P}_C$ such that $\frac{1}{2-\epsilon} \tilde{P}_C  \not \seq \S(\tilde{P}_C)$.
\end{proposition}

\begin{proof} 
Let $\epsilon >0$ and $n = \textup{max}\{2,\lceil \frac{1}{\epsilon}\rceil\}$.
Consider the instance 
$\textup{min} \{ \sum_{j = 1}^n x_j \,|\, x \in \tilde{P}_C\}$, where 
\begin{align*}
\tilde{P}_C = \{x \in \R^n_+ \,|\, x_i + \displaystyle \sum_{j \in [n]\setminus \{i\}} 2x_j \geq 2, \ \forall i \in [n]\}.
\end{align*}
Note that $\tilde{P}_C$ is well-behaved.
It is sufficient to show that 
$\frac{z^{\S}}{z^{LP}} \geq 2 - \epsilon$
for this instance. 

\begin{enumerate}
\item 
$z^{LP} \le \frac{2n}{2n - 1}$: 
It can be checked that the point $\bar x_j = \frac{2}{2n - 1}$ for each $j \in [n]$ is in 
$\tilde{P}_C$. Thus, $z^{LP} \le \frac{2n}{2n - 1}$. 
\item 
$z^{\S} \geq 2$: 
Adding all the constraints defining $\tilde{P}_C$ we obtain the valid inequality
\begin{eqnarray*} 
\sum_{j \in [n]} x_j \geq \frac{2n}{2n - 1}.
\end{eqnarray*}  
The corresponding CG cut is $\sum_{j \in [n]} x_j \geq 2$. Since each CG cut is also a split cut we obtain $z^{\S} \geq 2$.
\end{enumerate}
Thus, 
$\frac{z^{\S}}{z^{LP}} \geq 2 - \frac{1}{n}$; and our choice of $n$ completes the proof. 
\end{proof}
\begin{proposition}
\label{prop:CoveringSplitWell}
Let  $P_C$ be well-behaved, i.e., $A_{ij} \leq b_i$ for all $i \in [m], j \in [n]$. Then, $\S(P_C)$ is well-behaved.
\end{proposition}
\begin{proof}
Let $\bpar^\top x \geq \delta$ be a facet-defining (i.e., nondominated) inequality for $\S(P_C)$. If it is a CG cut, then we are done. Thus, we assume that it is a non-CG cut. For a contradiction, suppose that $\bpar_1 > \delta$. {\cred This is because any CG cut can be written as $\sum_j \lceil \sum_{i = 1}^m \lambda_i A_{ij} \rceil x_j \geq \lceil \sum_{i = 1}^m \lambda_i b_i \rceil$, where $\lambda_i \geq 0 $ for $i \in [m]$. Therefore $\sum_{i = 1}^m \lambda_i A_{ij} \leq \sum_{i = 1}^m \lambda_i $ for all $j \in [n]$, which implies that $\lceil \sum_{i = 1}^m \lambda_i A_{ij} \rceil \leq \lceil \sum_{i = 1}^m \lambda_i b_i \rceil$ for all $j \in [n]$.} 

We know that $\bpar^\top x \geq \delta$ is valid for $P_C^{\pi,\pi_0}$ for some $(\pi,\pi_0) \in \Z^n \times \Z$. Then, there exist multipliers $\lambda^1, \lambda^2 \in \R_+^m, \ \mu_1, \mu_2 \in \R_+$ and $\sigma^1, \sigma^2 \in \R_+^n$ such that
\begin{align}
\label{eqn:CovWell_SideOne}
& (\bpar,\delta) = \lambda^1 (A,b) + \mu_1 (-\pi,-\pi_0) + \sigma^1 (\ones,0) \\
\label{eqn:CovWell_SideTwo}
& (\bpar,\delta) = \lambda^2 (A,b) + \mu_2 (-\bar{\pi},-\bar{\pi}_0) + \sigma^2 (\ones,0)
\end{align}
where $(\bar{\pi},\bar{\pi}_0) = (-\pi,-\pi_0-1)$. Note that if $\sigma^1_1 > 0$ and $\sigma^2_1 > 0$, then we can obtain another split cut by decreasing both $\sigma^1_1$ and $\sigma^2_1$ by $\min \{ \sigma^1_1,\sigma^2_1\}$, which dominates the given split cut $\bpar^\top x \geq \delta$. Therefore, we assume, WLOG, that $\sigma^2_1 = 0$. Then, we make the following two cases: \\*[0.2cm]
\emph{Case 1}. $- \bar{\pi}_0 \geq - \bar{\pi}_1$: This implies that $\lambda^2 A_{\cdot 1} + \mu_2 (-\bar{\pi}_1) \leq \lambda^2 b + \mu_2 (-\bar{\pi}_0)$ (where $A_{\cdot 1}$ denotes the first column of $A$), equivalently $\bpar_1 \leq \delta$, which is a contradiction.
\\*[0.2cm]
\emph{Case 2}. $- \bar{\pi}_0 \ {\cred <} - \bar{\pi}_1$: This condition is equivalent to $1-\pi_1 < -\pi_0$. We first claim that $\sigma^1_1 > \mu_1$. From \eqref{eqn:CovWell_SideOne} and $\bpar_1 > \delta$, we have 
$$\lambda^1 A_{\cdot 1} - \mu_1 \pi_1 + \sigma^1_1 > \lambda^1 b + \mu_1 (-\pi_0) >  \lambda^1 b + \mu_1 (1-\pi_1),$$
which implies that 
$$(\lambda^1 A_{\cdot 1}-\lambda^1 b) + \sigma^1_1 - \mu_1 > 0.$$
As $P_C$ is well-behaved, we have $\lambda^1 A_{\cdot 1}-\lambda^1 b \leq 0$, thus we get $\sigma^1_1 > \mu_1$. Next, we let 
$$\tilde{\pi} := \pi - e_1, \ \tilde{\sigma}^1 := \sigma^1 - \mu_1 e_1, \ \tilde{\sigma}^2 := \sigma^2 + \mu_2 e_1.$$
Note that $\tilde{\sigma}^1_1 > 0$. Also, due to \eqref{eqn:CovWell_SideOne} and \eqref{eqn:CovWell_SideTwo}, we have
\begin{equation} 
\label{eq:CovWellBothSides}
(\bpar,\delta) = \lambda^1 (A,b) + \mu_1 (-\tilde{\pi},-\pi_0) + \tilde{\sigma}^1 (\ones,0) = \lambda^2 (A,b) + \mu_2 (\tilde{\pi},-\bar{\pi}_0) + \tilde{\sigma}^2 (\ones,0). 
\end{equation}
Note that $\mu_2 > 0$ since otherwise, i.e., when $\mu_2 = 0$, the equation \eqref{eqn:CovWell_SideTwo} and $\sigma^2_1 = 0$ give the contradiction $\bpar_1 = \lambda^2 A_{\cdot 1} \leq \lambda^2 b = \delta$. Therefore, we have $\tilde{\sigma}^2_1 > 0$ as well. If we reduce both $\tilde{\sigma}^1_1$ and $\tilde{\sigma}^2_1$ by a sufficiently small $\epsilon > 0$, so that they are still nonnegative, from \eqref{eq:CovWellBothSides} we obtain another valid split cut $(\bpar - 2 \epsilon e_1)^\top x \geq \delta$ which dominates $\bpar^\top x \geq \delta$, hence a contradiction.
\end{proof}

\textbf{Acknowledgements.} 
Santanu S. Dey would like to acknowledge the support of the NSF grant CMMI\#1149400. Marco Molinaro would like to acknowledge the support of CNPq grants Universal \#431480/2016-8 and Bolsa de Produtividade em Pesquisa \#310516/2017-0; this work was partially while the author was a Microsoft Research Fellow at the Simons Institute for the Theory of Computing. {\cred We would like to thank the reviewers for their careful and constructive comments that have significantly improved the paper.}
\bibliographystyle{plain}
\bibliography{LatticeFreeRank}   
\newpage
\begin{appendix}
\section{Additional proofs}
\label{sec:appendix}
\begin{observation} 
\label{obs:PackingSplitAgg}
For packing polyhedra, split cuts are not necessarily aggregation cuts.
\end{observation}
\begin{proof}
An example, where there exists a split cut that cannot be obtained as an aggregation cut, is provided in Figure \ref{fig:PackingSplitAgg}. 
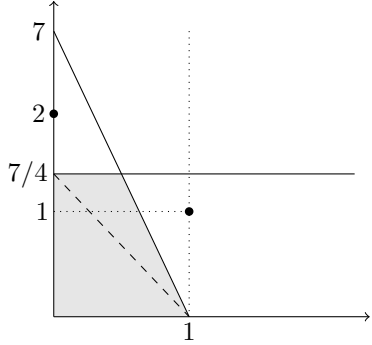
\begin{figure}[h]
\centering
\begin{tikzpicture}[scale=1] 
    \draw[white,fill=gray!20!white] (0,0) -- (1.8,0) -- (0.9,1.9) -- (0,1.9) -- cycle;
    \draw[->] (0,0) -- (0,4.2);
    \draw[->] (0,0) -- (4.2,0);
    \draw (0,3.8) -- (1.8,0);
    \draw (0,1.9) -- (4,1.9);
    \node[fill,circle,inner sep=1.15pt] at (0,2.7) {};
    \node[fill,circle,inner sep=1.15pt] at (1.8,1.4) {};
    \node (A) at (-0.2,3.8) {\small 7};
    \node (A) at (-0.2,2.7) {\small 2};
    \node (A) at (-0.35,1.9) {\small 7/4};
    \node (A) at (1.8,-0.2) {\small 1};
    \draw[dotted] (0,1.4) -- (1.8,1.4);
    \node (A) at (-0.152,1.4) {\small 1};
    \draw[dotted] (1.8,3.8) -- (1.8,0);
    \draw[dashed] (0,1.9) -- (1.8,0);
\end{tikzpicture}
\caption{A packing polyhedron for which there exists a split cut that cannot be obtained as an aggregation cut.}
\label{fig:PackingSplitAgg}
\end{figure}
In the figure, the shaded region represents the packing polyhedron $P = \{ x \in \R_+^2   \mid 7 x_1 + x_2 \leq 7, 4 x_2 \leq 7\}$. It is easy to see that $7 x_1+4 x_2 \leq 7$ (the dashed line in the figure) is a split cut obtained by using the split set $S(e_1,0)$. Note that this cut separates both of the points $(0,2)$ and $(1,1)$. We next show that this cut is not an aggregation cut by proving that $(0,2)$ and $(1,1)$ are not separated at the same time by any aggregation cut. An inequality is an aggregation cut for $P$ if it is valid for the set $P(\alpha) := \conv ( \{ x\in \R_+^2 \mid (7-7\alpha) x_1 + (3 \alpha +1) x_2 \leq 7 \})$ for some $\alpha \in [0,1]$. It can be easily verified that if $\alpha \leq 5/6$, then $(0,2) \in P(\alpha)$, and if $\alpha \geq 1/4$, then $(1,1) \in P(\alpha)$.
\end{proof}
\begin{observation} 
\label{obs:CoveringSplitAgg}
For covering polyhedra, split cuts are not necessarily aggregation cuts.
\end{observation}
\begin{proof}
An example, where there exists a split cut that cannot be obtained as an aggregation cut, is provided in Figure \ref{fig:CoveringSplitAgg}. 
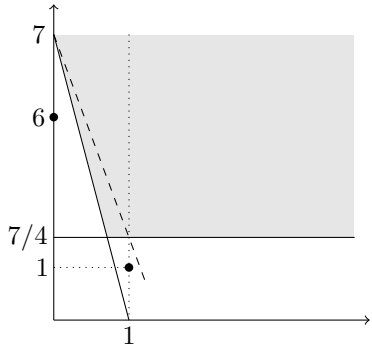
\begin{figure}[h]
\centering
\begin{tikzpicture}[scale=1] 
    \draw[white,fill=gray!20!white] (0,3.8) -- (0.7108,1.1) -- (4,1.1) -- (4,3.8) -- cycle;
    \draw[->] (0,0) -- (0,4.2);
    \draw[->] (0,0) -- (4.2,0);
    \draw (0,3.8) -- (1,0);
    \draw (0,1.1) -- (4,1.1);
    \node[fill,circle,inner sep=1.15pt] at (0,2.7) {};
    \node[fill,circle,inner sep=1.15pt] at (1,0.7) {};
    \node (A) at (-0.2,3.8) {\small 7};
    \node (A) at (-0.2,2.7) {\small 6};
    \node (A) at (-0.35,1.1) {\small 7/4};
    \node (A) at (1,-0.2) {\small 1};
    \draw[dotted] (0,0.7) -- (1,0.7);
    \node (A) at (-0.152,0.7) {\small 1};
    \draw[dotted] (1,3.8) -- (1,0);
    \draw[dashed] (0,3.8) -- (1.222,0.5); 
\end{tikzpicture}
\caption{A covering polyhedron for which there exists a split cut that cannot be obtained as an aggregation cut.}
\label{fig:CoveringSplitAgg}
\end{figure}
In the figure, the shaded region represents the covering polyhedron $P = \{ x \in \R_+^2   \mid 7 x_1 + x_2 \geq 7, 4 x_2 \geq 7\}$. It is easy to see that $21 x_1+4 x_2 \geq 28$ (the dashed line in the figure) is a split cut obtained by using the split set $S(e_1,0)$. Note that this cut separates both of the points $(0,6)$ and $(1,1)$. We next show that this cut is not an aggregation cut by proving that $(0,6)$ and $(1,1)$ are not separated at the same time by any aggregation cut. An inequality is an aggregation cut for $P$ if it is valid for the set $P(\alpha) := \conv ( \{ x\in \R_+^2 \mid (7-7\alpha) x_1 + (3 \alpha +1) x_2 \geq 7 \})$ for some $\alpha \in [0,1]$. It can be easily verified that if $\alpha \geq 1/18$, then $(0,6) \in P(\alpha)$, and if $\alpha \leq 1/4$, then $(1,1) \in P(\alpha)$.
\end{proof}
\end{appendix}

\end{document}